\documentclass[a4paper]{amsart}
\usepackage{amsthm,amsmath,amssymb,mathrsfs,url,amsaddr}
\theoremstyle{definition}
\newtheorem{theorem}{Theorem}[section]
\newtheorem*{theorem*}{Theorem}
\newtheorem{lemma}[theorem]{Lemma}
\newtheorem{proposition}[theorem]{Proposition}
\newtheorem{corollary}[theorem]{Corollary}
\newtheorem{definition}[theorem]{Definition}
\newtheorem{remark}[theorem]{Remark}

\newcommand\DN{\newcommand}

\DN\ts{\times}
\DN\ots{\otimes}
\DN\bra{\langle}
\DN\ket{\rangle}
\DN\ms{\medskip}
\DN\st{\,;\,}

\DN\lref[1]{Lemma~\ref{#1}}
\DN\tref[1]{Theorem~\ref{#1}}
\DN\pref[1]{Proposition~\ref{#1}}
\DN\sref[1]{Section~\ref{#1}}
\DN\dref[1]{Definition~\ref{#1}}
\DN\rref[1]{Remark~\ref{#1}} 
\DN\corref[1]{Corollary~\ref{#1}}

\newcommand{\ep}{\varepsilon}

\DN{\limi}[1]{\lim_{#1\to\infty}} 	
\DN\PD[2]{\frac{\partial #1 }{\partial #2}}
\DN\map[3]{#1:#2 \to #3}

\DN\R{\mathbb{R}}\DN\Rd{\mathbb{R}^d}
\DN\N{\mathbb{N}}
\DN\Q{\mathbb{Q}}
\DN{\Z}{\mathbb{Z}}
\DN\C{\mathbb{C}}
\DN\F{\mathbb{F}}
\DN\HH{\mathbb{H}}

\DN\xx{\mathbf{x}}
\DN\XX{\mathbf{X}}
\DN\yy{\mathbf{y}}
\DN\zz{\mathbf{z}}

\DN\WN{W_{}^N}
\DN\WNp{W_{}^{N+1} }
\DN\intnWN{\mathring{W}_{\ge}^N}
\DN\intWNp{\mathring{W}_{}^{N+1}}
\DN\WNNp{W^{N,N+1}}
\DN\nWNN{W_{\ge}^{N,N}}
\DN\nWN{W_{\ge}^{N}}
\DN\nWNp{W_{\ge}^{N+1}}

\DN\LtNNp{\Lambda_{\theta, N}^{N+1}}
\DN\LtaNNp{\Lambda_{\theta, \alpha, N}^{N+1}}
\DN\LtaNN{\Lambda_{\theta, \alpha, N}^{N}}

\DN\diag{\mathrm{diag}}
\DN\rad{\mathfrak{rad}}
\DN\eval{\mathfrak{eval}}
\DN\sym{\mathrm{sym}}
\DN\Lag{\mathrm{Lag}}

\date{\today}
\begin{document}
\title[]{The intertwining property for $\beta $-Laguerre processes and integral operators for Jack polynomials }

\author{KAWAMOTO Yosuke}
\address{Graduate school of environmental, life, natural science and technology, Okayama University, Okayama, Japan}
\email{y-kawamoto@okayama-u.ac.jp}

\author{SHIBUKAWA Genki}
\address{Kitami Institute of Technology, Hokkaido, Japan}
\email{g-shibukawa@mail.kitami-it.ac.jp}


\maketitle

\begin{abstract}
The aim of this paper is to study intertwining relations for Laguerre process with inverse temperature $\beta \ge 1$ and parameter $\alpha >-1$.
We introduce a Markov kernel that depends on both $\beta $ and $ \alpha $, and establish new intertwining relations for the $\beta$-Laguerre processes using this kernel.
A key observation is that Jack symmetric polynomials are eigenfunctions of our Markov kernel, which allows us to apply a method established by Ramanan and Shkolnikov.
Additionally, as a by-product, we derive an integral formula for multivariate Laguerre polynomials and multivariate hypergeometric functions associated with Jack polynomials.
\end{abstract}

\tableofcontents

\section{Introduction}
\subsection{Markov kernels and Laguerre processes}

Our purpose is to find intertwining relations for Laguerre processes for general inverse temperature $\beta$.
Whilst intertwining relations for the $\beta $-Laguerre process were shown in \cite{Ass19}, we establish new relations by using different Markov kernels, which are a $\beta$-extension of the kernel used in \cite{BuK}.

Throughout this paper, we set $\theta = \beta / 2$.
Set the Weyl chamber $\WN =\{ \mathbf{x}_N =(x_1,\ldots,x_N) \in \R^N  \st x_1\le \ldots \le x_N  \}$.
We will write $\xx_N $ simply $\xx $ when no confusion can arise.
The internal of $\WN $ is denoted by $\mathring {W}^{N} := \{\xx\in \WN \st x_1 < \cdots < x_{N} \}$.
We write the set of interlaced configuration as 
\begin{align*}
  \WNNp  (\xx )& :=\{\yy \in \WN \st x_1 \le y_1\le x_2\le \ldots \le y_N \le x_{N+1} \} \quad\text{ for }\xx =(x_1,\ldots, x_{N+1}) \in \WNp 
.\end{align*}  
The following formula is useful to introduce Markov kernels.

\begin{lemma}\cite[Propositon 4.2.1]{For10} \label{l:12}
Let $( w_1,\ldots, w_{N+1} )$ be a random variable distributed according to the Dirichlet distribution with parameters $(s_1, \ldots, s_{N+1})$ and let $\xx \in \mathring{W}^{N+1}  $.
Then, the roots $\yy \in \WN $ of the random rational function 
\begin{align*}
 y \mapsto \sum_{i=1}^{N+1} \frac{w_i}{y-x_i} \bigg(= \frac{ \sum_{i=1}^{N+1}\big( w _i \prod_{1\le j \le N+1, j\neq i} (y -x_j) \big)}{\prod _{i=1}^{N+1} (y-x_i)} \bigg)
\end{align*}
has the probability density function
\begin{align*}
  \frac{\Gamma(s_1+\cdots +s_{N+1})}{\Gamma(s_1) \cdots \Gamma(s_{N+1})}  \frac{\prod_{1\le i<j\le N } (y_j-y_i)}{\prod_{1\le i <j\le N+1} (x_j-x_i )^{s_j+s_i-1}} \prod_{i=1}^{N+1} \prod_{j=1}^{N} |x_i-y_j|^{s_i -1} \mathbf{1}_{\WNNp (\xx)}(\yy) 
.\end{align*}
\end{lemma}

Following \cite{AsN21}, we introduce the Markov kernel $\LtNNp $ via a random polynomial.
\begin{definition}\label{d:12}
Let $\theta>0$ and  $\xx=(x_1,\ldots, x_{N+1}) \in \WNp $, let $\LtNNp (\xx, \cdot )$ be the distribution on $\WN $ of the roots of the random polynomial 
  \begin{align*}
  y\mapsto  \sum_{i=1}^{N+1} \Big( w _i \prod_{1\le j \le N+1, j\neq i} (y -x_j) \Big)
  ,\end{align*}
where $(w_1,\ldots, w _{N+1}) $ is distributed according to the Dirichlet distribution with all parameters being equal to $\theta$.
\end{definition}

The kernel $\LtNNp $ is an integral operator on $\WN $ depending on a parameter $\xx \in \WNp$.
In this sense, we write $\LtNNp : \WNp \dashrightarrow \WN $.
We can describe the probability density of the kernel $\LtNNp $ due to \lref{l:12}.
\begin{proposition}\label{p:13}
For any $\xx \in \mathring{W}^{N+1} $, we have
\begin{align}\label{:11a}
  \LtNNp (\xx, d\yy ) &= \frac{1}{Z_{\theta }^{N}} \frac{\prod_{1\le i<j\le N } (y_j-y_i) }{\prod_{1\le i<j\le {N+1}} (x_j-x_i)^{2\theta-1}} \prod_{i=1}^{N+1} \prod_{j=1}^{N}|x_i-y_j|^{\theta -1} \mathbf{1}_{\WNNp (\xx)}(\yy) d\yy 
,\end{align}
where the normalisation constant is given by 
\begin{align*}
  \frac{1}{Z_{\theta }^{N} } :=\frac{\Gamma((N+1)\theta)}{\Gamma(\theta)^{N+1}}
.\end{align*}
\end{proposition}

The distribution \eqref{:11a} is referred to as the Dixon-Anderson conditional distribution because this probability density was found in the work by Dixon \cite{Dix05}, and considered by Anderson in the context of the Selberg integral \cite{And91}. 
The formulation of Dixon--Anderson distribution by \dref{d:12} is used in \cite{AsN21}. 

We discuss intertwining relations between Markov kernels and the semigroups of stochastic processes related to random matrices.
Several diffusions intertwined by $\LtNNp $ have been identified.
The first example intertwined by this kernel is the non-intersecting Brownian motions, also known as the Dyson Brownian motions for $\beta=2$ \cite{War07}.
Other examples for $\beta=2$ were studied in \cite{Ass20, AOW19}.
For general $\beta$, the intertwining relations for the $\beta $-Dyson Brownian motions were established \cite{GoS15,RaS18}, and those for the $\beta$-Laguerre and $\beta$-Jacobi processes were obtained \cite{Ass19}.

Here, we focus on intertwining relations for the $\beta$-Laguerre processes.
Set $\nWN = \WN \cap [0,\infty )^N$.
The $N$-dimensional Laguerre process for the inverse temperature $\beta \ge 1 $ is the process on $\nWN $ defined by the stochastic differential equation  
\begin{align}\label{:11d}
  dX_t^{N,i} &=\sqrt{2 X_t^{N, i} } dB_t^i +\frac{\beta }{2 }\Big( \alpha +1 +\sum_{1\le j\le N, j\neq i}\frac{2 X_t^{N, i} }{X_t^{N, i} -X_t^{N,j} } \Big)dt, \qquad 1\le i\le N 
  ,\end{align}
where $(B^i)_{i=1}^N $ is the $N$-dimensional Brownian motion and $\alpha > -1$.
For any starting point $\xx_N \in\nWN$, the equation \eqref{:11d} has a unique strong solution with no collisions and no explosions \cite[Corollary 6.5]{GrM14}.
The Laguerre processes for general $\beta$ were introduced and studied in \cite{Dem07b}.
When $\beta=1$, the Laguerre process is the squared singular values process of a matrix with Brownian entries \cite{Bru91}.
The Laguerre process for $\beta=2 $ is identical to the non-colliding square Bessel process, which describes the squared singular values process of a matrix with complex Brownian entries \cite{Dem07, KoO01}.

Set $\theta = \beta /2$ as before.
Let $\{ T_{\theta, \alpha ,t}^{ N} \}_{t\ge 0 }$ be the Markov semigroup associated with the solution to \eqref{:11d}.
The following intertwining relation for Laguerre processes was established in \cite{Ass19} (see also \cite[Section 3.7]{AOW19} for $\theta=1$). 
\begin{proposition}[\cite{Ass19}]\label{p:14}
Assume that $\theta \ge 1/2$ and $\alpha >-1$.  
Then, for any $N\in\N$ and $t\ge 0$, we have the equality of Markov kernels
\begin{align}\label{:11f}
T_{\theta, \alpha ,t}^{N+1 } \LtNNp  =\LtNNp T_{\theta, \alpha +1 ,t}^{N } 
.\end{align}
\end{proposition}
Remark that, in \eqref{:11a}, the $\alpha$-parameters of the Laguerre processes on the left- and right-hand sides are different.
For this reason, we refer to the equation \eqref{:11f} as the shifted intertwining relation.

In this paper, we establish intertwining relations with fixed parameters, analogous to those for the Ornstein--Uhrenbeck counterpart of the Laguerre process for $\theta =1$ \cite{BuK}.
Specifically, we will identify a kernel $\LtaNNp $, depending on $\theta $ and $\alpha $, such that 
\begin{align}\label{:11g}
  T_{\theta, \alpha ,t}^{N+1 }   \LtaNNp  =\LtaNNp T_{\theta, \alpha  ,t}^{ N}
.\end{align}
Unlike in \eqref{:11f}, the  $\alpha$-parameters on both sides of \eqref{:11g} are the same.

\subsection{Main results}
To obtain new intertwining relations for the Laguerre processes, we introduce a kernel $\LtaNN : \nWN  \dashrightarrow \nWN $ depending on $\theta $ and $\alpha $.

\begin{definition}\label{d:13}
Suppose $\theta>0$ and $\alpha >-1$.
For $\xx=(x_1,\ldots, x_{N}) \in \nWN $, let $\LtaNN (\xx, \cdot )$ be the distribution on $\nWN $ of the root of the random polynomial
\begin{align*}
y \mapsto  \sum_{i=0}^{N} \Big( w _i \prod_{0 \le j \le N, j\neq i} (y -x_j) \Big)
,\end{align*}
where $(w_0, w_1,\ldots, w _{N}) $ is distributed according to the Dirichlet distribution with parameter $(s_0, s_1, \ldots, s_N)$ given by $s_0=(\alpha+1) \theta$ and $s_1=\cdots =s_N=\theta $.
Here, we use the symbol $x_0=0$ for notational convenience. 
 \end{definition}
  
The probability density of the kernel $\LtaNN $ is also explicitly given from \lref{l:12}.
We set $\mathring {W}_{\ge}^{N} = \{\xx\in \nWN \st x_1 < \cdots < x_{N} \}$ and  
\begin{align*}
\nWNN (\xx )& :=\{\yy \in \nWN \st 0\le  y_1 \le x_1\le y_2\le \ldots \le y_N \le x_{N} \} \quad\text{ for }\xx =(x_1,\ldots, x_{N}) \in \nWN
.\end{align*}
\begin{proposition}\label{p:15}
Suppose $\theta \ge 1/2 $ and $\alpha >-1$.
For any $\xx \in \intnWN $, we have  
\begin{align}\label{:11b}
\LtaNN (\xx, d\yy) &=\frac{1}{Z_{\theta,\alpha }^{N}}  \frac{\prod_{1\le i<j\le N } (y_j-y_i) }{\prod_{1\le i<j\le N } (x_j-x_i) ^{2\theta-1}} \prod_{i,j=1}^{N}|x_i-y_j|^{\theta -1} \prod_{i=1}^{N} \frac{ y_i^{\theta(\alpha+1)-1 }}{ x_i^{\theta(\alpha+2) -1}  }  \mathbf{1}_{\nWNN (\xx)}(\yy) d\yy
,\end{align}
where the normalisation constant is given by 
\begin{align*}
  \frac{1}{Z_{\theta,\alpha }^{N}} :=\frac{\Gamma((N+\alpha+1)\theta )}{\Gamma(\theta )^N \Gamma ((\alpha+1) \theta )} 
 .\end{align*} 
\end{proposition}

When $\alpha \in \{0\} \cup \N $, the probability density \eqref{:11b} for $\theta=1/2, 1, 2$ was established as the distribution of the squared singular values of a fixed matrix multiplied with a truncated orthogonal, unitary, symplectic random matrix, respectively \cite{AhS22, KKS16}.
This fact will be used in \sref{s:6}.

For $\theta>0 $ and $\alpha >-1$, we define the Markov kernel $\LtaNNp : \nWNp \dashrightarrow \nWN $ as  
\begin{align}\label{:11c}
\LtaNNp  =\LtNNp \LtaNN
.\end{align}
The kernel $\LtaNNp $ gives an intertwining relation for $\beta$-Laguerre processes with a fixed parameter $\alpha$.

\begin{theorem}\label{t:16}
Suppose $\theta \ge 1/2 $ and $\alpha >-1$.
Then, for any $ N\in\N$ and $ t\ge0 $, we have the equality of Markov kernels 
\begin{align} \label{:16a}
T_{\theta, \alpha, t}^{N+1} \LtaNNp = \LtaNNp  T_{\theta, \alpha, t}^{N} 
.\end{align}  
\end{theorem}



\tref{t:16} is demonstrated by employing new shifted intertwining relations. 

\begin{theorem}\label{t:17}
Suppose $\theta \ge 1/2 $ and $\alpha >-1$.
Then, for any $ N\in\N$ and $ t\ge 0$,  we have the shifted intertwining relation
\begin{align}\label{:17b}
  T_{\theta, \alpha+1, t}^{N} \LtaNN &= \LtaNN  T_{\theta, \alpha, t}^{N} 
  .\end{align}  
\end{theorem}
 
Assuming \tref{t:17}, we can easily prove \tref{t:16} as follows.

\medskip
\noindent
{\em Proof of \tref{t:16}  } \quad 
The equations \eqref{:11f} and \eqref{:17b} imply 
\begin{align}\label{:17c}
  T_{\theta, \alpha, t}^{N+1} \LtNNp \LtaNN &= \LtNNp T_{\theta, \alpha+1, t}^{N} \LtaNN = \LtNNp \LtaNN  T_{\theta, \alpha, t}^{N} 
.\end{align}
Therefore, combining \eqref{:17c} with \eqref{:11c}, we complete the proof.
\qed
\ms

Thus, it remains to establish \tref{t:17}, and we apply the method used in \cite{RaS18} to achieve this .
The first step in their approach involves establishing intertwining relations at the level of generators by using Jack symmetric polynomials.
We introduce key relations for Jack polynomials in the following subsection.

\subsection{Integral operators and Jack polynomials}
In what follows, let $\lambda =(\lambda_1 \ge \lambda_2 \ge \ldots) $ be a partition of an integer and define $l(\lambda )$ as the length of $\lambda $.
Let $P_\lambda(\xx_N ; \theta )=x_1^{\lambda_1}\cdots x_{N}^{\lambda_N}+ \cdots$ be the Jack symmetric polynomial parametrised by a partition $\lambda$.
See \sref{s:2} for the precise definition of Jack polynomials.
Hereafter, we use $P_{\lambda }^\theta (\xx_N ):=P_{\lambda } (\xx_N ;\theta )$ for simplicity. 

For $r\in\R$, the shifted factorial is defined as $(x)_r =\Gamma (x+r)/\Gamma(x)$.
We set
\begin{align}
c(\lambda, N,\theta \,; \alpha) &:=\frac{\Gamma((N+\alpha+1)\theta) }{\Gamma ( (\alpha+1 )\theta ) } \prod_{i=1}^{N} \frac{\Gamma((N+\alpha+1-i)\theta +\lambda _i )}{\Gamma ( (N+\alpha+2-i)\theta +\lambda _i)}
\\& \notag
=\prod_{i=1}^{N} \frac{((N+\alpha+1-i)\theta)_{\lambda _i }}{ ( (N+\alpha+2-i)\theta)_{\lambda _i} }
.\end{align}
Note that $c(\lambda, N,\theta \,; \alpha)\in \mathbb{Q}(\theta, \alpha)$ for fixed $\lambda$ and $N $.
Furthermore, we use $c(\lambda , N, \theta ):=c(\lambda, N,\theta \,; 0) $.

Jack polynomials satisfies the following relation in terms of $\LtNNp $, which is the key to showing the intertwining relation at the level of generators in \cite{Ass19,RaS18}.
\begin{lemma} \label{l:21} \cite{Oko98, OkO97}
Suppose $\theta>0$.
Then, for any partition $\lambda $ with $l(\lambda )\le N$ and $\xx_{N+1} \in \WNp $, we have
  \begin{align}\label{:22a}
   [\LtNNp P_{\lambda }^\theta ](\xx_{N+1} ) &=c(\lambda , N, \theta ) P_{\lambda }^\theta  (\xx_{N+1} ) 
,\end{align}  
where we write $ [\LtNNp P_{\lambda }^\theta ](\xx_{N+1} ) := \int P_{\lambda }^\theta (\yy_N ) \LtNNp (\xx_{N+1}, d \yy_N )$.
\end{lemma}

The equation \eqref{:22a} suggests that Jack polynomials are eigenfuncions of $\LtNNp $.
However, strictly speaking, this is not correct: whilst $P_{\lambda }^\theta$ on the left-hand side of \eqref{:22a} is in $N$ variables, on the right-hand side, it is in $N+1$ variables.
On the other hand, we establish that Jack polynomials are indeed eigenfunctions of $\LtaNN $.

\begin{theorem}\label{t:22}
Suppose $\theta >0$ and $\alpha> -1$.
Then, for any $\lambda $ with $l(\lambda )\le N$ and $\xx_{N} \in \nWN $, we have
\begin{align} \label{:22b}
 [\LtaNN P_{\lambda}^\theta]  (\xx_N) =   c(\lambda, N ,\theta \,; \alpha ) P_{\lambda }^\theta (\xx_N)
.\end{align}
\end{theorem}

Furthermore, using \tref{t:22}, we can compute the action of $\LtaNN $ on multivariate Laguerre polynomials and multivariate hypergeometric functions, as demonstrated in the appendix.


\subsection{Organisation of this paper}
The present paper is organised as follows.
We collect some facts about Jack polynomials in \sref{s:2}.
In \sref{s:3}, we give the proof of \tref{t:22}.
\sref{s:5} is devoted to the proof of \tref{t:17}.
In \sref{s:6}, we observe that the kernel $\LtaNNp $ has an interpretation in terms of squared singular values of invariant random matrices for classical parameter $\theta =1/2,1, 2 $.
In Appendix \ref{appendix:A}, we apply \tref{t:22} to obtain a formula for symmetric functions.
In Appendix \ref{appendix:B}, we prove that Ornstein-Uhlenbeck counterparts of Laguerre processes satisfy intertwining relations similar to \tref{t:16} and \tref{t:17}.

\section{Preliminaries on Jack symmetric polynomials}\label{s:2}

We begin by introducing Jack polynomials.
See \cite[Chapter 2]{ViK95} for further details.
We introduce the differential operators in $N$-variables
\begin{align*}
E_k ^N &= \sum_{i=1}^{N} x_i^k \PD{}{x_i}, 
\\
D_k ^N &= \sum_{i=1}^{N} x_i^k \frac{\partial ^2 }{ \partial x_i^2 } + 2\theta \sum_{i=1}^{N} \sum_{1\le j\le N, j\neq i} \frac{x_i^k}{x_i-x_j} \PD{}{x_i}
.\end{align*}

Let $P_\lambda^ \theta  (\xx_N ) :=P_\lambda(\xx_N ; \theta )=x_1^{\lambda_1}\cdots x_{N}^{\lambda_N}+ \cdots $ be the Jack polynomial parametrised by a partition $\lambda $.
By definition, $P_\lambda ^\theta (\xx_N )$ is the symmetric polynomial eigenfunction of $D_2^N $: 
\begin{align}\label{:31a}
D_2^N  P_\lambda^\theta  (\xx_N ) =  e(\lambda, N, \theta) P_\lambda ^\theta (\xx_N )
,\end{align}
where 
\begin{align*}
  e(\lambda, N, \theta) :=2 B(\lambda' ) -2\theta B(\lambda ) +2\theta(N-1)|\lambda|  
.\end{align*}
Here, $|\lambda| =\sum \lambda _i$, $B(\lambda)=\sum (i-1)\lambda_i $, and $\lambda' $ is the conjugate partition of $\lambda $.
The eigenfunction is unique up to normalisation, and we take normalisation such that the leading coefficient of $P_\lambda ^\theta $ is $1$.
Specifically, 
\begin{align}\label{:20a}
  P_\lambda^\theta(1_{N} ) = \prod_{1\le i <j \le N}(\lambda_i-\lambda_j +\theta(j-i))_\theta \prod_{k=1}^{N} \frac{\Gamma(\theta)}{\Gamma(\theta k)}
,\end{align}
where $1_N $ denotes the $N$-dimensional vector whose components are all $1$ \cite[(6.4)]{OkO97}.

From \cite[(2.13a), (2.13b), (2.13d)]{BaF97}, Jack polynomials satisfy the following relations:
\begin{align} \label{:40a}
  E_0^N  P_\lambda ^\theta (\xx_N  ) &=P_{\lambda} ^\theta (1_N ) \sum_{i=1}^{l (\lambda) } \binom{\lambda}{ \lambda_{(i)}} _\theta
  \frac{P_{\lambda_{(i)}}^\theta  (\xx_N  )}{P_{\lambda_{(i)}}^\theta  (1_N  )}
  ,\\\label{:40c}
   E_1^N P_\lambda^\theta  (\xx_N ) &=|\lambda| P_\lambda ^\theta  (\xx_N  )
   ,\\\label{:40b}
  D_1^N P_\lambda^\theta  (\xx_N  )&= P_{\lambda}^\theta (1_N) \sum_{i=1}^{l(\lambda) } 
  \binom{\lambda}{ \lambda_{(i)}} _\theta
  (\lambda_i-1+(N-i)\theta )\frac{P_{\lambda_{(i)}}^\theta  (\xx_N  )}{P_{\lambda_{(i)}}^\theta  (1_N  )}
.\end{align}
Here, $\lambda_{(i)}$ denotes the partition given by $\lambda_{(i)} =(\lambda_1,\ldots, \lambda_{i-1}, \lambda_{i} -1, \lambda_{i+1 }, \cdots)$, and $ \binom{\lambda }{\rho }_\theta $ is the generalised binomial coefficient introduced in \cite{Las90}, determined by the expansion
\begin{align}\label{:20b}
  \frac{P_{\lambda }^\theta  (1_N+\xx_N  )}{P_{\lambda }^\theta  (1_N )}=\sum_{m=0}^{|\lambda| } \sum_{|\rho|=m}  
  \binom{\lambda }{\rho }_\theta 
   \frac{P_{\rho }^\theta  (\xx_N )}{P_{\rho  }^\theta  (1_N  )}
.\end{align}

\section{Integral operators for Jack polynomials} \label{s:3}
We now prove \tref{t:22} using \lref{l:21}.

\medskip 
\noindent
{\em Proof of \tref{t:22}} \quad 
Remark that it is sufficient to show \eqref{:22b} for any $\xx_N \in \mathring{W}_{\ge}^N $.
Actually, by the definition \dref{d:13} and the fact that the roots of a polynomial are continuous function of its coefficients, we have $\limi{n} \LtaNN (\xx_N^n, \cdot) = \LtaNN (\xx_N, \cdot ) $ weakly if $\limi{n} \xx_N ^n =\xx_N$. 
Furthermore, since $\LtaNN (\xx_N^n, \cdot)$ is support compact, we obtain $\limi{n} [ \LtaNN P_{\lambda }^{\theta}] (\xx_N^n) =[\LtaNN P_{\lambda }^{\theta}] (\xx_N)$.

Let us first show \eqref{:22b} under the assumption that $\theta \alpha $ is a non-negative integer.
Let $\xx_N =(x_1,\ldots, x_N) \in \mathring{W}_{\ge}^N $.
Then, applying \eqref{:22a} for $( 0, \xx _{N}) \in \mathring{W}^{N+1} $, we obtain
\begin{align} \label{:22c}
  [\LtNNp P_{\lambda +(\theta \alpha) 1_{N} }^\theta ] (( 0, \xx _{N}) ) & =c(\lambda + (\theta\alpha) 1_{N}, N, \theta ) P_{\lambda + (\theta \alpha) 1_{N} } ^\theta (( 0, \xx _{N})  )
  \\ & \notag
=c(\lambda + (\theta \alpha) 1_{N}, N, \theta )  \Big( \prod_{i=1}^{N } x_ i^{\theta \alpha}  \Big) P_{\lambda }^\theta  (\xx_{N} ) 
.\end{align}  
Here, we used the formula $P_{\lambda + 1_{N}}^\theta  (\xx_N  ) =\big( \prod_{i=1}^{N}x_i \big) P_{\lambda} ^\theta (\xx_N )$ \cite[(1.2)]{OkO97}.
On the other hand, the left-hand side of \eqref{:22c} turns out to be
\begin{align}&\label{:22d}
  \frac{1}{Z_{\theta }^{N} }  \int_0^{x_1} \ldots \int_{x_{N-1}}^{x_N} \frac{\prod_{1\le i<j\le N } (y_j-y_i) }{\prod_{1\le i<j\le N } (x_j-x_i) ^{2\theta-1}} \prod_{i,j=1}^{N} |x_i-y_j|^{\theta -1}  \prod_{i=1}^{N} \frac{ y_i^{\theta -1}}{x_i^{2\theta-1 }}  P_{\lambda +(\theta \alpha )1_{N}} ^\theta (\yy ) d\yy
\\& \notag
=  \frac{1}{Z_{\theta }^{N}}  \int_0^{x_1} \ldots \int_{x_{N-1}}^{x_N}  \frac{\prod_{1\le i<j\le N } (y_j-y_i) }{\prod_{1\le i<j\le N } (x_j-x_i) ^{2\theta-1}}\prod_{i,j=1}^{N} |x_i-y_j|^{\theta -1}  \prod_{i=1}^{N}\frac{y_j^{\theta(\alpha +1) -1}}{ x_i^{2\theta-1 }} P_{\lambda  }^\theta  (\yy ) d\yy
\\& \notag 
=   \frac{Z_{\theta, \alpha}^{N}}{Z_{\theta }^{N}} \Big( \prod_{i=1}^{N } x_ i^{\theta \alpha} \Big)  [\LtaNN  P_{\lambda  }^\theta ] (\xx_N  )
\end{align}
from \eqref{:11a} and \eqref{:11b}.
Combining \eqref{:22c} and \eqref{:22d}, we have  
\begin{align}\label{:22e}
  \LtaNN P_{\lambda}^\theta  (\xx_N) = \frac{ Z_{\theta }^{N}  c(\lambda +(\theta\alpha) 1_N, N, \theta ) }{Z_{\theta, \alpha }^{N}} P_{\lambda }^\theta  (\xx_N   )
.\end{align}
Thus, the proof of \eqref{:22b} is finished if $\theta \alpha $ is a non-negative integer.

We now turn to the case of general $\theta >0 $ and $\alpha >-1$.
For $\theta $ fixed, we set the function for $z \in \C$ 
\begin{align*}
  f(z)&:= \int_0^{x_1} \ldots \int_{x_{N-1}}^{x_N}  \frac{\prod_{1\le i<j\le N } (y_j-y_i) }{\prod_{1\le i<j\le N } (x_j-x_i) ^{2\theta-1}} \prod_{i,j=1}^{N}|x_i-y_j|^{\theta -1}  \prod_{i=1}^{N}\frac{ y_i^{z+ \theta -1}}{ x_i^{z+ 2\theta-1 }} P_{\lambda  }^\theta  (\yy  ) d\yy
  \\& \qquad
   -Z_{\theta}^{N} c(\lambda +z 1_{N}, N, \theta )  P_{\lambda }^\theta  (\xx_N   )
.\end{align*}
The equation \eqref{:22e} implies $f(n)=0$ for any non-negative integer $n$.
The function $f(z )$ is analytic on $\Re z >0 $ and continuous on $\Re z \ge 0$.
Furthermore, for some constants $C, \tau >0$ and $c <\pi$,  we have
\begin{align*}
  |f(z) | \le C e^{\tau |z|} \text{ for any }\Re z >0
  , \\
  |f(iy ) | \le C e^{c |y|}\text{ for any }y\in \R 
.\end{align*}
Here, to estimate $c(\lambda +z 1_{N}, N, \theta )$, we have used  $\Gamma (z+a) / \Gamma (z+b) \sim  z^{a-b}$ as $z \to \infty $ in $|\arg z |\le\pi -\delta$ .
Therefore, the Carlson theorem implies $f(z)=0$ for $\Re z \ge 0$.
Because $f$ is analytic on  $\Re z >-\theta $, we have $f\equiv 0$ on $\Re z >-\theta $.
Thus, we obtain $f(\theta \alpha )=0$ for any $\theta>0$ and $\alpha>-1$, which complete the proof of \tref{t:22}.
\qed
\ms

We show some examples of \tref{t:22}.
When $N=1$, for a non-negative integer $\lambda $, the equation \eqref{:22b} is rephrased as 
\begin{align}\label{:22h}
  \frac{\Gamma ((\alpha+2)\theta )}{\Gamma(\theta) \Gamma((\alpha+1 )\theta)}   \int_0^{x} (x-y)^{\theta -1}  \frac{ y^{\theta(\alpha+1)-1 }}{ x^{\theta(\alpha+2) -1}  } \frac{ y^\lambda }{( (\alpha +1)\theta )_{\lambda} } dy = \frac{ x^\lambda }{( (\alpha +2)\theta )_{\lambda} }
\end{align}
from $P_{\lambda }^\theta (x) =x^\lambda $.
The equation \eqref{:22h} can be shown also by straightforward computation.

When $N=2$ and $\lambda =(\lambda_1, \lambda_2)$, the equation \eqref{:22b} is written as
\begin{align}\label{:22f}
  \frac{\Gamma ((\alpha+3)\theta )}{\Gamma(\theta)^2 \Gamma((\alpha+1 )\theta)} \int_{0}^{x_1} \int_{ x_1}^{x_2} \frac{  (y_2-y_1) }{(x_2-x_1) ^{2\theta-1}} \prod_{i,j=1}^{2}|x_i-y_j|^{\theta -1}  \frac{ (y_1 y_2)^{\theta(\alpha+1)-1 }}{ (x_1 x_2)^{\theta(\alpha+2)-1}  } P_{\lambda }^\theta (y_1,y_2)  dy_1 dy_2 
 \\ \notag
  =\prod_{i=1}^{2} \frac{((\alpha+3-i)\theta)_{\lambda _i }}{ ( (\alpha+4-i)\theta)_{\lambda _i} } P_{\lambda }^\theta (x_1, x_2)
.\end{align}
The Jack polynomial in two-variables has the following representation \cite[(10.15)]{Koo15}: 
\begin{align*}
P_{\lambda }^\theta (x_1, x_2) &=x_1^{\lambda_1} x_2 ^{\lambda _2} {}_{2}F_{1} \Big( \begin{matrix}
-\lambda_1 + \lambda_2, \, \theta
\\
1-\lambda_1+\lambda_2 -\theta 
\end{matrix}  \,;\, \frac{x_2}{x_1} \Big)
\\&
=\frac{(\lambda_1 -\lambda_2)!}{(\theta)_{\lambda_1 -\lambda_2 }}(x_1 x_2)^{\frac{\lambda_1+\lambda_2}{2}} C_{\lambda_1-\lambda_2}^\theta \bigg( \frac{x_1 +x_2}{2(x_1 x_2)^{\frac{1}{2}}} \bigg)
,\end{align*}
where ${}_{2}F_{1}$ is the hypergeometric function and $C_{m}^\theta $ is the Genenbauer polynomial.
Therefore, the equation \eqref{:22f} gives an integral representation of these special functions. 
To the best of our knowledge, a direct derivation of the equation \eqref{:22f} using hypergeometric functions or the Gegenbauer polynomials is not known.

The equation \eqref{:22b} for a rectangular partition $\lambda =(m,\cdots, m)$ with $l(\lambda)=N$ is 
\begin{align}&
\frac{\Gamma(\theta(N+\alpha+1))}{\Gamma(\theta )^N \Gamma (\theta(\alpha+1))} \int_{0}^{x_1} \cdots \int_{x_{N-1}}^{x_N}  \frac{\prod_{1\le i<j\le N } (y_j-y_i) }{\prod_{1\le i<j\le N } (x_j-x_i) ^{2\theta-1}} \prod_{i,j=1}^{N}|x_i-y_j|^{\theta -1} \prod_{i=1}^{N} \frac{ y_i^{\theta(\alpha+1)-1+m }}{ x_i^{\theta(\alpha+2) -1 +m }  }   d\yy 
\\&\notag
=\prod_{i=1}^{N}  \frac{((N+\alpha+1-i)\theta)_{m } }{ ( (N+\alpha+2-i)\theta)_{m} }  
,\end{align}
which is just a paraphrase of $ \int \Lambda_{\theta, \alpha+m/\theta, N }^N (\xx, d\yy)=1$.

\section{Proof of \tref{t:17}}\label{s:5}
We observed the actions of differential operators $E_k^N, D_k^N$ on Jack polynomials in \sref{s:2}.
Combining these with \tref{t:22}, we obtain intertwining relations between these differential operators and $\LtaNN $.

\begin{lemma}\label{l:32}
Assume that  $\theta >0$ and $\alpha>-1$.
Then, for any partition $\lambda $ with $l(\lambda )\le N$, the following holds for $\xx_N\in \nWN $:
  \begin{align}
    \label{:31b}
    [ E_1^N  \LtaNN P_{\lambda }^\theta ]  (\xx_N  ) &=[ \LtaNN E_1^N P_{\lambda }^\theta ] (\xx_N )  
    \\  
  \label{:41a}
  [\big( D_1^N + \theta(\alpha +2)E_0^N \big) \LtaNN P_{\lambda}^\theta ] ( \xx_N ) &=[ \LtaNN \big(  D_1^N + \theta(\alpha +1)E_0^N \big) P_{\lambda}^\theta ] ( \xx_N )
.\end{align}

\begin{proof}
The equation \eqref{:31b} results from \eqref{:22b} and \eqref{:40c}.
From \eqref{:22b}, \eqref{:40a}, and \eqref{:40b}, we have the equations 
  \begin{align}& \label{:41b}
  [\big( D_1^N + \theta(\alpha +2)E_0^N \big)\LtaNN P_{\lambda}^\theta ] ( \xx_N )
   \\& \notag
   = c (\lambda,N, \theta \,; \alpha )P_{\lambda}^\theta (1_N )  \bigg[ \sum_{i=1}^{l(\lambda )} 
   \begin{pmatrix}
   \lambda \\ \lambda_{(i)}  
   \end{pmatrix}_\theta
   \big\{(N-i + \alpha+2)\theta + \lambda_i-1 \big\}\frac{P_{\lambda_{(i)}}^\theta  (\xx_N  )}{P_{\lambda_{(i)}}^\theta  (1_N )} \bigg]
,\\&
\label{:41c}
[\LtaNN \big( D_1^N + \theta(\alpha +1)E_0^N \big)  P_{\lambda}^\theta ] ( \xx_N  )
\\& \notag 
= P_{\lambda}^\theta (1_N )  \sum_{i=1}^{l(\lambda )} 
  \begin{pmatrix}
  \lambda \\ \lambda_{(i)}  
  \end{pmatrix}_\theta
  c(\lambda_{(i)}, N, \theta \,; \alpha )\big\{(N-i+\alpha+1)\theta + \lambda_i-1 \big\}\frac{P_{\lambda_{(i)}}^\theta  (\xx_N  )}{P_{\lambda_{(i)}}^\theta  (1_N )}
.\end{align}
A direct computation yields, for any $i$, 
\begin{align*}
  c(\lambda_{}, N , \theta \,; \alpha )\big\{(N-i+\alpha+2)\theta + \lambda_i-1 \big\} =c(\lambda_{(i)}, N, \theta \,; \alpha )\big\{(N-i+\alpha+1)\theta + \lambda_i-1 \big\}  
,\end{align*}
which implies the coefficients of $P_{\lambda_{(i)}}$ in \eqref{:41b} and \eqref{:41c} are the same.
Thus, we have proved \eqref{:41a}.

\end{proof}
\end{lemma}

Let $A_{\theta, \alpha }^{N} $ be the generator associated with the stochastic differential equations \eqref{:11d}.
Specifically, 
\begin{align}\label{:32a} 
  A_{\theta, \alpha }^{N} &=  D_1^N + \theta (\alpha +1 ) E_0^N 
.\end{align}

\begin{lemma}\label{l:33}
  Assume that  $\theta >0$ and $\alpha>-1$.
Then, for any partition $\lambda $ with $l(\lambda ) \le N$, the following holds for any $\xx_N\in \nWN $:
\begin{align} \label{:33a}
  [A_{\theta, \alpha +1}^{N} \LtaNN P_{\lambda}^\theta ] (\xx_{N})&:= [ \LtaNN A_{\theta, \alpha}^N P_{\lambda}^\theta ] (\xx_{N})
.\end{align}
\begin{proof}
The equation \eqref{:33a} results from \eqref{:31b} and \eqref{:41a} with \eqref{:32a}.
\end{proof}
\end{lemma}

We next show that the intertwining relation for generators \lref{l:33} gives rise to that for semigroups \tref{t:17}.
The following argument is almost identical to that of \cite{Ass19,RaS18}; however, for the conveniense of the reader, a concise proof will be provided.

We write $\XX^N $ as the solution to \eqref{:11d}.
An exponential moment estimate of $\XX^N $ follows from Step 2 in \cite[pp.1888--1889]{Ass19}.
Thus, for some $\ep >0$, we have $E_{\xx_N} [e^{\ep ||\XX_t^N ||}] <\infty$ for any $t\ge0$, where $||\cdot || $ denotes the $l_1$-norm on $\R^N$.

For a fixed partition $\lambda $, define $L( \{P_{\kappa}^\theta \}_{\kappa \subseteq \lambda } )$ as the finite dimensional vector space spanned by $\{P_{\kappa}^\theta \}_{\kappa \subseteq \lambda }$.
The kernel $\LtaNN $ can be regarded as a operator on $L( \{P_{\kappa}^\theta \}_{\kappa \subseteq \lambda } )$ by \tref{t:22}.
Thus, $\LtaNN $ acts on the space as a matrix, denoted by  $M_1:=[M_1( \kappa , \nu )]_{ \kappa ,\nu \subseteq \lambda }$ .
Then, by the definition of $M_1$, we have 
\begin{align}\label{:51b}
\LtaNN P_{\kappa} ^\theta (\xx_N )  =\sum_{\nu \subseteq  \lambda } M_1(\kappa ,\nu ) P_{\nu}^\theta (\xx_N )  
\end{align}
for any $\kappa \subseteq \lambda $ and $\xx_N \in\nWN $.
Similarly, from \lref{l:21}, the kernel $\LtNNp $ acts on $L( \{P_{\kappa}^\theta \}_{\kappa \subseteq \lambda } )$ as a matrix, denoted by $M_1'$. 

\begin{lemma}\label{l:51}
Assume that $\theta \ge 1$ and $\alpha>-1$.
Then, for any partition $\lambda $ with $l(\lambda) \le N $, the following holds for any $ \xx_N \in W_{\ge}^{N}$:
\begin{align}\label{:41g}
   [ T_{\theta, \alpha+1 ,t}^{N } \LtaNN P_{\lambda  }^{\theta } ] (\xx_N )= [ \LtaNN T_{\theta, \alpha, t}^N P_{\lambda }^\theta]  (\xx_N ) 
   .\end{align}
\begin{proof}
Because of \eqref{:32a} with \eqref{:40a} and \eqref{:40b}, the generator $A_{\theta, \alpha }^{N}$ can be regarded as a operator on  $L( \{P_{\kappa}^\theta \}_{\kappa \subseteq \lambda } )$.
We define $M_2 :=[M_2(\kappa ,\nu )]_{\kappa ,\nu \subseteq \lambda }$ as the matrix describing an action of $A_{\theta, \alpha }^{N} $ on this vector space.
That is, the matrix $M_2 $ is determined by 
\begin{align}\label{:52a}
A_{\theta, \alpha }^{N} P_\kappa ^\theta =\sum_{\nu \subseteq \lambda } M_2(\kappa ,\nu ) P_\nu^\theta   
\end{align}
for any $ \kappa \subseteq \lambda$.
Similarly, let $M_3$ be the matrix describing an action of $A_{\theta, \alpha +1 }^{N} $ on  $L( \{P_{\kappa}^\theta \}_{\kappa \subseteq \lambda } )$.
Then, we derive $ M_1 M_3 =M_2M_1 $ from \eqref{:33a}.
Therefore, for any $t\ge 0$, we have
  \begin{align}\label{:41f}
    M_1 e^{t M_3}  = e^{tM_2}  M_1
. \end{align}

By the It\^o formula to compute $P_{\lambda }^\theta (\XX^N  ) $ and taking expectation, we have
\begin{align}\label{:41d}
  T_{\theta, \alpha, t}^{N} P_{\lambda }^\theta (\xx_N ) = P_{\lambda } ^\theta (\xx_N ) + \int_0^t T_{\theta, \alpha, s}^{N} A_{\theta, \alpha }^{N} P_{\lambda } ^\theta (\xx_N ) ds
\end{align}  
for any $\xx_N \in W_{\ge}^{N} $.
Fix $\xx_N \in W_{\ge}^{N} $ and set $f_{\kappa }(t) :=T_{\theta, \alpha, t }^{N} P_{\kappa}^\theta (\xx_N ) $.
Because \eqref{:41d} holds for any $\kappa $ with $\kappa \subseteq \lambda$, we obtain 
\begin{align*}
  f_{\kappa } (t) =f_{\kappa}(0) + \sum_{\nu \subseteq \lambda }M_2(\kappa, \nu )\int_0^t f_{\nu}(s) ds
\end{align*}
from \eqref{:52a}, which has a unique solution 
\begin{align}\label{:41e}
  f_{\kappa }(t)=\sum_{\nu \subseteq \lambda } e^{t M_2}(\kappa, \nu) f_{\nu} (0)
.\end{align}
From \eqref{:41f} with \eqref{:41e} and its analogue for $M_3$ instead of $M_2$, we obtain \eqref{:41g}.
\end{proof}  
\end{lemma}

A probability measure $m$ on $\nWN $ gives rise to a symmetrised probability measure $m^{\sym }$ on $[0,\infty)^N $ via 
\begin{align*}
  m^{\sym } (dz_1,\ldots , dz_{N})= \frac{1}{N!} m (dz_{(1)},\ldots ,dz_{(N)})
,\end{align*}
where $z_{(1)}\le z_{(2)} \le \cdots \le z_{(N)}$ are the ordered statistics of $(z_1, z_2,\ldots, z_N)$.
Furthermore, for a (not necessarily symmetric) function $f$ on $ [0,\infty)^N $, we have
\begin{align}\label{:42b}
  \int_{[0,\infty)^N } f(\zz ) m^\sym( d\zz) =\int_{[0,\infty)^N} f^\sym (\zz ) m^\sym(d \zz) =  \int_{\nWN } f^\sym (\zz ) m(d \zz)
,\end{align}
where $f^\sym (\zz )$  is a symmetric function given by $f^\sym (\zz )=\frac{1}{N!} \sum_{\sigma\in S_{N}}f(z_{\sigma(1) },\ldots, z_{\sigma(N) })$.

\bigskip 
\noindent
{\em Proof of \tref{t:17}} \quad 
Fix $\xx \in W_{\ge}^N $. 
Recall that Jack polynomials form a basis for the space of symmetric polynomials.
Therefore, from \eqref{:41g}, we obtain
\begin{align}\label{:42c}
  [ T_{\theta, \alpha+1 ,t}^{N } \LtaNN p ] (\xx_N )= [ \LtaNN T_{\theta, \alpha, t}^N p ]  (\xx_N ) 
\end{align}
for any symmetric polynomial $p$ in $N$-variables.
Therefore, from \eqref{:42b} and \eqref{:42c}, all moments of the symmetrised measures of $ [ T_{\theta, \alpha+1 ,t}^{N } \LtaNN  ] (\xx_N , \cdot ) $ and $[ \LtaNN T_{\theta, \alpha, t}^N ]  (\xx_N , \cdot )  $ are the same.
From \cite[Theorem 1.3]{deJ03} with the fact that $[\LtaNN e^{\ep ||\cdot || }](\zz) \le e^{\ep || z || } $ and the exponential moments estimate of $\XX_N$, we see that the symmetrised versions of $ [ T_{\theta, \alpha+1 ,t}^{N } \LtaNN  ] (\xx_N , \cdot ) $ and $[ \LtaNN T_{\theta, \alpha, t}^N p ]  (\xx_N , \cdot )  $ are the same probability measure on $[0,\infty)^N$.
Thus, these measures coincide on $\nWN $, and we conclude the theorem.
\qed
\ms

\section{Interpretation of Markov kernels for classical $\theta $ in terms of the radial parts of random matrices} \label{s:6}
In this section, we focus on the classical parameter $\theta=1/2,1,2$.
Furthermore, we assume that $\alpha $ is a non-negative integer.
For these specific values of $\theta $ and $ \alpha$, the kernel $\LtaNNp $ has interpretation in the context of random matrix theory, as described in \tref{t:12}.
This interpretation was originally established for $\theta=1$ in \cite{BuK}, and the proof for $\theta =1/2, 2$ follows by a nearly identical argument.
For reader's convenience, we will briefly outline the proof.

\subsection{Conditional eigenvalues of invariant random matrices}
Let $\F $ denote $\R$, $\C$, or the skew field of quaternions $\HH $, which corresponds to $\theta=1/2,1,2$, respectively. 
Let $M_{m,n}(\F )$ be the space of $m\times n$ matrices over $\F $, and for brevity write $M_{n}(\F )=M_{n,n}(\F )$.
Define the subset $H_{n}(\F ) \subset M_{n}(\F ) $ as the space of real symmetric matrices for $\F=\R$, complex Hermite matrices for $\F= \C$, or quaternion self-dual matrices for $\F=\HH $.
Furthermore, define $\mathbb{U}_{n }(\F ) \subset M_{n}(\F ) $ as the space of orthogonal matrices for $\F=\R$, unitary matrices for $\F=\C$, or symplectic matrices for $\F=\HH $.

For $m_1\ge m_2, n_1\ge n_2 $, let $\pi_{m_2, n_2}^{m_1, n_1} : M_{m_1, n_1}  (\F ) \to M_{m_2,n_2} (\F ) $ be the natural projection sending an $m_1 \times n_1 $ matrix to its upper left $m_2 \times n_2$ corner.
We employ the expression $\pi_{m_2,n_2}^{m_1}$ in place of $\pi_{m_2,n_2}^{m_1,n_1}$ if $m_1=n_1$, and use a similar symbol for $m_2=n_2$.

We define a map $\mathfrak{eval}_{n} : H_{n}(\F ) \to W^n  $ as 
\begin{align*}
\mathfrak{eval}_{n} (X)=(\lambda_{1}(X), \ldots , \lambda_{n}(X))
,\end{align*}
where $(\lambda_{i}(X))_{i=1}^n$ is the eigenvalues of $X$ arranged in non-decreasing order.
Furthermore, define the radial part $\mathfrak{rad}_{n} : M_{m,n} (\F ) \to W_{\ge }^n $ as $\mathfrak{rad}_{n} (X)=\mathfrak{eval} _n(X^{*}X)$.
Let a probability measure $P_{\mathfrak{eval}}^n [X ] $ on $W^n $ be the distribution of the eigenvalues of a random matrix $X \in H_{n}(\F ) $.
Similarly, let a probability measure $P_{\mathfrak{rad}}^n [X]  $ on $W_{\ge}^n$ denote the distribution of the radial part of a random matrix $X\in M_{m,n}(\F ) $.

Let $U_{N+1}\in \mathbb{U}_{N+1}$ be a Haar distributed random matrix.
For $\xx \in W^{N+1}$, let $\mathrm{diag}(x_1,\ldots, x_{N+1}) $ denote  the square matrix of order $N+1$ with deterministic diagonal elements given by $\xx $.
Then, the equality 
\begin{align}\label{:51e}
  \LtNNp (\xx, \cdot ) =P_{\eval }^N [ \pi_N^{N+1} (U_{N+1}^* \diag(x_1,\ldots, x_{N+1}) U^{N+1} ) ] 
\end{align}
holds for any  $\xx \in W^{N+1}$.
The equation \eqref{:51e} was proved for $\theta= 1 $ in \cite[Proposition 4.2]{Bar01}, and extended to $\theta =1/2,2$ in \cite{Ner03} (see also \cite[Proposition 1.7]{AsN21}).
As an immediate consequence of \eqref{:51e}, we obtain the following lemma:
\begin{lemma}
If a random matrix $X_{N+1}\in H_{N+1}(\F )$ is $\mathbb{U}_{N+1}(\F ) $-invariant by conjugation in the sense that $U_{N+1}^{*} X_{N+1} U_{N+1} \stackrel{\mathrm{law}}{=} X_{N+1}$ for any $U_{N+1} \in \mathbb{U}_{ N+1}( \F ) $, then we have the equality of probability measures
\begin{align}\label{:51c}
P_{\mathfrak{eval}}^{N+1}[X_{N+1}] \LtNNp=P_{\mathfrak{eval}}^{N}[ \pi_{N}^{N+1}( X_{N+1}) ]
.\end{align}
\end{lemma}

\subsection{Interpretation of $\LtaNNp $ }
In this subsection, $\alpha $ is supposed to be a non-negative integer.
We begin by explaining the following interpretation of $\LtaNN $ in the context of random matrix theory.
This interpretation is established for $\theta =1 $ in \cite{KKS16}, and extended to $\theta =1/2,2$ in \cite{AhS22}.
\begin{lemma} \label{l:15}
Assume that $\alpha$ is a non-negative integer.
Let $V_{N+\alpha +1} \in \mathbb{U}_{N+\alpha+1}(\F )$ be a Haar distributed random  matrix.
Then, for any $\zz= (z_1 ,\ldots, z_N) \in \nWN $, we have
\begin{align}\label{:53b}
\LtaNN (\zz ,\cdot)= P_{\mathfrak{rad}}^{N} [ \pi_{N+\alpha, N}^{N+\alpha+1}( V_{N+\alpha+1} ) \mathrm{diag}(\sqrt{z_1},\ldots, \sqrt{z_{N}}) ] 
.\end{align}

\begin{proof}
It is sufficient to show \eqref{:53b} for $\zz\in\intnWN$.
Applying \cite[Theorem 2.7]{AhS22} with the setting $(m,l, n, \nu)=(N+\alpha+1, N, N,\alpha)$ yields that, for $\zz\in \intnWN $ satisfying $0<z_1 < \cdots <z_N <1$, the probability density of 
\begin{align*}
P_{\mathfrak{rad}}^{N} [\pi_{N+\alpha, N}^{N+\alpha+1}( V_{N+\alpha+1} )  \mathrm{diag}(\sqrt{z_1},\ldots, \sqrt{z_{N}})  ]
\end{align*}
is given by \eqref{:11a}, and this can be easily extended to $\zz\in \intnWN $.
Thus, we obtain \eqref{:53b} for $\zz\in\intnWN $, which completes the proof. 
\end{proof}
\end{lemma}

We say a random matrix $X_{m,n}\in M_{m,n}(\F )$ is $\mathbb{U}_{m}(\F ) \times \mathbb{U}_{n}(\F ) $-invariant if $X_{m,n}\stackrel{\mathrm{law}}{=} V_{m} X_{m,n} U_{n}$ for any fixed matrices $V_{m}\in \mathbb{U}_{m} (\F ), U_{n}\in \mathbb{U}_{n} (\F )$.
The following theorem was proved for $\theta =1$ in \cite{BuK}.

\begin{theorem}\label{t:12}
  Let $\alpha$ be a non-negative integer.
  For an $\mathbb{U}_{N+\alpha+1}(\F ) \times \mathbb{U}_{ N+1}(\F )$-invariant random matrix $X_{N+\alpha+1,N+1} \in M_{N+\alpha+1,N+1}(\F )$, let $X_{N+\alpha, N}:=\pi_{N+\alpha,N}^{N+\alpha+1, N+1}( X_{N+\alpha+1,N+1})$ be its truncation.
  Then, we have the equality of probability measures
  \begin{align*}
  P_{\mathfrak{rad}}^{N+1} [ X_{N+\alpha+1,N+1} ] \LtaNNp =P_{\mathfrak{rad}}^{N} [X_{N+\alpha, N}]
  .\end{align*}
 \begin{proof}
Setting 
$X_{N+\alpha+1,N} :=\pi_{N+\alpha+1, N}^{N+\alpha+1, N+1} (X_{N+\alpha+1,N+1} )$, we have $$\pi_{N}^{N+1}( X_{N+\alpha+1,N+1}^{*} X_{N+\alpha+1, N+1} )=X_{N+\alpha+1,N}^{*} X_{N+\alpha+1,N}.$$
Furthermore, because $ X_{N+\alpha+1,N+1}^{*} X_{N+\alpha+1, N+1} \in H_{N+1}(\mathbb{F })$ is $\mathbb{U}_{N+1}(\F) $-invariant by conjugation, the equation \eqref{:51c} yields
\begin{align}\label{:12c}
P_{\mathfrak{rad}}^{N+1} [ X_{N+\alpha+1,N+1} ]  \LtNNp =  P_{\mathfrak{rad}}^{N} [ X_{N+\alpha+1,N} ] 
.\end{align}

For a random variable $(z_1,\ldots, z_N)$ distributed as $P_{\mathfrak{rad}} ^N[ X_{N+\alpha+1,N}] $, we set 
\begin{align*}
D_{N+\alpha+1,N}:=
\begin{bmatrix}
D_{N}  \\
\mathbf{0}_{(\alpha+1 )\times N}
\end{bmatrix}
, \quad  D_{N}:=\mathrm{diag}(\sqrt{z_1},\ldots,\sqrt{z_N})
.\end{align*}
Let $U_{N}\in \mathbb{U}_{N} (\F )$ and $V_{N+\alpha+1}\in \mathbb{U} _{N+\alpha+1}(\F )$ be Haar distributed random matrices such that $D_{N+\alpha+1,N}, U_{N}$, and $ V_{N+\alpha+1}$ are independent.
Then, for a similar reason as in \cite[Lemma 2.4]{Def10}, we obtain
\begin{align}\label{:12e}
X_{N+\alpha+1, N} \stackrel{\mathrm{law}}{=} V_{N+\alpha+1} D_{N+\alpha+1,N} U_N
.\end{align}
A direct computation with \eqref{:12e} yields
\begin{align*}
X_{N+\alpha, N} =\pi_{N+\alpha, N}^{N+\alpha+1, N} (X_{N+\alpha+1, N}) \stackrel{\mathrm{law}}{=} \pi_{N+\alpha,N}^{N+\alpha+1} (V_{N+\alpha+1} ) D_{N} U_N
.\end{align*}
Therefore, using this with the equation \eqref{:53b},  we have
\begin{align}\label{:12d}
P_{\mathfrak{rad}}^{N} [X_{N+\alpha, N}]  = P_{\mathfrak{rad}}^{N } [\pi_{N+\alpha,N}^{N+\alpha+1} (V_{N+\alpha+1} ) D_{N} U_{N}]= P_{\mathfrak{rad}}^{N} [ X_{N+\alpha+1,N} ] \LtaNN
.\end{align}
Collecting \eqref{:12c} and \eqref{:12d} with \eqref{:11c}, we complete the proof of \pref{t:12}.
\end{proof}
\end{theorem}

\begin{corollary}\label{c:16}
Let $V_{N+\alpha+1}  \in \mathbb{U}_{N+\alpha+1}(\F ) $ and $U_{N+1}  \in \mathbb{U}_{N+1}(\F ) $ be Haar distributed independent random matrices and $D_{N+\alpha+1, N+1} \in M_{N+\alpha+1, N+1}(\F )$ be a deterministic matrix given by
  \begin{align*}
  D_{N+\alpha+1,N+1}=
  \begin{bmatrix}
    \mathrm{diag}(\sqrt{x_1},\ldots,\sqrt{x_{N+1}} )  \\
  \mathbf{0}_{\alpha \times (N+1)}
  \end{bmatrix}
  \text{ for }\xx =(x_1,\ldots, x_{N+1})\in \nWNp
  .\end{align*} 
  Then, the probability measure $\LtaNNp (\xx , \cdot) $ is the same as 
  \begin{align*}
  P_{\mathfrak{rad}}^{N} \big[ \pi _{N+ \alpha ,N}^{N+\alpha+1, N+1}(V_{N+\alpha+1}  D_{N+\alpha+1, N+1} U_{N+1}) \big]
  .\end{align*}
  \begin{proof}
 Because $V_{N+\alpha+1}  D_{N+\alpha+1, N+1}  U_{N+1}$ is an $\mathbb{U}_{N+\alpha+1}(\F ) \times \mathbb{U}_{ N+1}(\F)$-invariant random matrix and $P_{\rad }^{N+1} [V_{N+\alpha+1}  D_{N+\alpha+1, N+1}  U_{N+1}] =\delta_{\xx} $ holds, the statement of this corollary follows from \tref{t:12}.
\end{proof}
\end{corollary}

Here we show an example obtained as a straightforward consequence of \tref{t:12}.
Let $X \in M_{N+\alpha +1, N+1} (\F ) $ be a random matrix whose probability density is proportional to 
\begin{align*}
 \exp \big( - \mathrm{Tr} X^* X \big) dX
.\end{align*}
The distribution of the squared singular values of $X$ is given by the Laguerre ensemble
\begin{align}
m_{\theta, \alpha }^N (d\xx ) := \frac{1}{\mathcal{Z}_{\theta, \alpha}^{N}}  \prod_{1\le i <j \le N} |x_j -x_i |^{2\theta} \prod_{i=1}^{N} x_i^{\theta (\alpha+1)-1} e^{-x_i}
,\end{align}
where $\mathcal{Z}_{\theta, \alpha}^N$ is the normalise constant.
Because $X$ is $\mathbb{U}_{m}(\F ) \times \mathbb{U}_{n}(\F ) $-invariant, we can use \tref{t:12}, and the resulting equation is 
\begin{align}
m_{\theta, \alpha}^{N+1} \LtaNNp =m_{\theta, \alpha}^{N} 
,\end{align}
for $\theta=1/2,1,2$ and $\alpha\in \{0\} \cup \N$.
\appendix

\section{Multivariate Laguerre polynomials and the integral operator $\LtaNN $}\label{appendix:A}

In this appendix, we investigate the action of the kernel $\LtaNN $ on symmetric functions other than Jack polynomials.
For any partition $\lambda $ with $l(\lambda ) \leq N$, we define multivariate Laguerre polynomial (or generalised Laguerre polynomial) $L_{\lambda }^{a}(\xx_{N} ;\theta ^{-1})$ by  
\begin{align} \label{:A1}
L_{\lambda }^{a}(\xx_{N} ;\theta ^{-1})
   :=
   \frac{1}{|\lambda |!}
   \sum_{\mu \subseteq \lambda }
      (-1)^{|\mu |}
      \binom{\lambda }{\mu }_{\theta }
      \prod_{i = 1}^{N}
         \frac{(a + 1 + \theta (N - i))_{\lambda_{i}}}{(a + 1 + \theta (N - i) )_{\mu_{i}}}
      \frac{P_{\mu }^{\theta }(\xx_N)}{P_{\mu }^{\theta}(1_N)}.
\end{align}
Our notation $L_{\lambda }^{a}(\xx_{N} ;\theta ^{-1})$ follows the Baker-Forrester style \cite{BaF97}. 
Thus, the parameter $\theta ^{-1}$ is not a typo. 
We also remark that the notation of the Jack polynomial in \cite{BaF97} is $C_{\lambda }^{\alpha }(\xx_{N})$ (this $\alpha $ is not our $\alpha$) and the correspondence between our notation and theirs is as follows:
$$
\frac{C_{\lambda }^{\theta ^{-1}}(\xx_{N})}{C_{\lambda }^{\theta ^{-1}}(1_{N})}
   =
   \frac{P_{\lambda }^{\theta }(\xx_{N})}{P_{\lambda }^{\theta }(1_{N})}.
$$

Originally, Laguerre polynomials were introduced as the polynomial eigenfunctions of the following Calogero-Surtherland type operator \cite{Las91}:
$$
\tilde{H}^{(L)}
   :=
   D_{1}^N + (a + 1)E_{0}^N - E_{1}^N.
$$
In fact, we have
\begin{align}\label{:A2}
\tilde{H}^{(L)}L_{\lambda }^{a}(\xx_{N} ;\theta ^{-1})
   =
   -|\lambda |L_{\lambda }^{a}(\xx_{N} ;\theta ^{-1}).
\end{align}

We remark that $\tilde{H}^{(L)} +E_1^N $  appears in the intertwining relation \eqref{:41a}, where $a:=\theta (\alpha +1)-1$.
As a generalisation of the equation \eqref{:41a}, we have the following lemma.
\begin{lemma}\label{c:43}
Suppose $\theta > 0$ and $\alpha > -1$. 
Let $\lambda $ be a partition with $l(\lambda )\le N$.
Then, for any $k\in\N$ and $\xx_N\in \nWN $, we have
  \begin{align}\label{:43a}
  [\big( D_1^N + \theta(\alpha +2)E_0^N \big)^k \LtaNN P_{\lambda}^\theta ] ( \xx_N ) &= [ \LtaNN \big(  D_1^N + \theta(\alpha +1)E_0^N \big)^k P_{\lambda}^\theta]( \xx_N )
. \end{align}
\begin{proof}
The proof is by induction on $k$.
The equation \eqref{:43a} for $k=1$ is exactly \eqref{:41a}.
Write $C_{\alpha} := D_1^N + \theta(\alpha +1)E_0^N$.
If we assume $(C_{\alpha +1 }^{k-1} \LtaNN -\LtaNN C_{\alpha }^{k-1} ) P_{\lambda}^\theta =0$ for some $k$, then we have
\begin{align*}&
(C_{\alpha +1 }^k \LtaNN -\LtaNN C_{\alpha }^k ) P_{\lambda}^\theta
\\& \notag
= C_{\alpha+1 }(C_{\alpha +1 }^{k-1} \LtaNN -\LtaNN C_{\alpha }^{k-1} ) P_{\lambda}^\theta + (C_{\alpha +1 }^{} \LtaNN -\LtaNN C_{\alpha }^{}  ) C_{\alpha}^{k-1} P_{\lambda}^\theta
\\& \notag
=0
.\end{align*}
Here, we use the fact that $C_{\alpha}^{k-1} P_{\lambda}^\theta$ is given by a linear combination of $P_{\mu } ^\theta $ with $l(\mu)\le N$ from \eqref{:40a} and \eqref{:40b}.
Thus, we establish the desired result.
\end{proof}
\end{lemma}

The multivariate Laguerre polynomial $L_{\lambda }^{a}(\xx_{N} ;\theta ^{-1})$ has the following Rodrigue type formula \cite[(4.39)]{BaF97} (Baker-Forrester said that ``prompted by M. Lassalle''):
\begin{align}\label{:A3}
\frac{(-1)^{|\lambda |}}{|\lambda |!}\exp{\big(-D_{1}^N - (a + 1)E_{0}^N \big)}\frac{P_{\lambda }^{\theta}(\xx_N)}{P_{\lambda }^{\theta }(1_N)}
 =
L^{a}_{\lambda }(\xx_{N};\theta ^{-1})
.\end{align}
From this Rodrigue type formula, \tref{t:22}, and \lref{c:43}, we calculate the action of $\LtaNN$ on $L_{\lambda }^{a}(\xx_{N} ;\theta ^{-1})$. 

\begin{theorem}
\label{thm:Laguerre intertwine}
Suppose $\theta > 0$ and $\alpha > -1$. 
Then, for any $\lambda $ with $l(\lambda )\leq N$ and $\xx_{N} \in \nWN $, we have
\begin{align}
\label{eq:Laguerre intertwine}
[\LtaNN L_{\lambda }^{a}(\cdot ;\theta^{-1} ) ] (\xx_N )
   =
   L_{\lambda }^{a + \theta }(\xx_{N} ;\theta^{-1} )c(\lambda, N,\theta \,; \alpha),
\end{align}
where $a := \theta (\alpha + 1) - 1$.
\end{theorem}
\begin{proof}
Since the degree of $P_{\lambda }^{\theta}(\xx_N)$ is $|\lambda |$ and 
$$
\big(-D_{1}^N - (a + 1)E_{0}^N \big)^{k}P_{\lambda }^{\theta}(\xx_N) = 0 \quad \text{for } k > |\lambda |,
$$
we have
\begin{align*}
L_{\lambda }^{a}(\xx_{N} ;\theta^{-1})
   &=
   \frac{(-1)^{|\lambda |}}{|\lambda |!}  \exp{\big(-D_{1}^N - (a + 1)E_{0}^N \big)}\frac{P_{\lambda }^{\theta}(\xx_N)}{P_{\lambda }^{\theta}(1_N)} \\
   &=
   \frac{(-1)^{|\lambda |}}{|\lambda |!} 
   \sum_{k = 0}^{|\lambda |}
       \frac{\big(-D_{1}^N - \theta (\alpha + 1)E_{0}^N \big)^{k}}{k!}
   \frac{P_{\lambda }^{\theta }(\xx_N)}{P_{\lambda }^{\theta }(1_N)}
\end{align*}
from \eqref{:A3}.
By applying \lref{c:43} and \tref{t:22}, we obtain
\begin{align*}
[\LtaNN L_{\lambda }^{a}(\cdot ;\theta^{-1})] (\xx_{N})
   &=
   \frac{(-1)^{|\lambda |}}{|\lambda |!} 
   \sum_{k = 0}^{|\lambda |}
       \frac{\LtaNN\big(-D_{1}^N - \theta (\alpha + 1)E_{0}^N \big)^{k}}{k!}
   \frac{P_{\lambda }^{\theta }(\xx_N)}{P_{\lambda }^{\theta }(1_N)} \\
   &=
   \frac{(-1)^{|\lambda |}}{|\lambda |!} 
   \sum_{k = 0}^{|\lambda |}
       \frac{\big(-D_{1}^N - \theta (\alpha + 2)E_{0}^N \big)^{k}\LtaNN}{k!}
   \frac{P_{\lambda }^{\theta }(\xx_N)}{P_{\lambda }^{\theta }(1_N)} \\
   &=
   \frac{(-1)^{|\lambda |}}{|\lambda |!} 
   \sum_{k = 0}^{|\lambda |}
       \frac{\big(-D_{1}^N - (a  + \theta + 1)E_{0}^N \big)^{k}}{k!}
   \frac{P_{\lambda }^{\theta }(\xx_N)}{P_{\lambda }^{\theta }(1_N)}
   c(\lambda, N ,\theta \,; \alpha ) \\
   &=
   L_{\lambda }^{a + \theta }(\xx_{N} ;\theta^{-1} )
   c(\lambda, N,\theta \,; \alpha).
\end{align*}
\end{proof}

In the above proof, we use the intertwining relation \eqref{:43a}. 
On the other hand, we give another proof of (\ref{eq:Laguerre intertwine}) without \lref{c:43}.
In fact, \tref{t:22} is equivalent to the following formula:
\begin{align}
\label{:A5}
[\LtaNN P_{\lambda }^{\theta }](\xx_{N})
   \prod_{i = 1}^{N}  \frac{1}{((N + \alpha + 1 - i) \theta)_{\lambda_{i}}}
    =
    P_{\lambda }^{\theta }(\xx_{N})
    \prod_{i=1}^{N}
    \frac{1}{((N + \alpha + 2 - i) \theta)_{\lambda_{i}}}.
\end{align}
Then, from the definition \eqref{:A1} and the formula \eqref{:A5}, we have the conclusion \eqref{eq:Laguerre intertwine}:
\begin{align*}&
[\LtaNN L_{\lambda }^{a}(\cdot  ;\theta^{-1})](\xx_{N})
   \\&=
   \frac{1}{|\lambda |!}
   \sum_{\mu \subseteq \lambda }
      (-1)^{|\mu |}
      \binom{\lambda }{\mu }_{\theta }
      \prod_{i = 1}^{N}
         \frac{((N + \alpha + 1 - i) \theta )_{\lambda_{i}}}{((N + \alpha + 1 - i) \theta )_{\mu_{i}}}  \frac{[\LtaNN P_{\mu }^{\theta }](\xx_N)}{P_{\mu }^{\theta}(1_N)}\\
   &=
   \frac{1}{|\lambda |!}
   \sum_{\mu \subseteq \lambda }
      (-1)^{|\mu |}
      \binom{\lambda }{\mu }_{\theta }
      \prod_{i = 1}^{N}
         \frac{((N + \alpha + 1 - i) \theta )_{\lambda_{i}}}{((N + \alpha + 2 - i) \theta )_{\mu_{i}}}    \frac{P_{\mu }^{\theta }(\xx_N)}{P_{\mu }^{\theta}(1_N)}
\\
   &=
   \frac{1}{|\lambda |!}
   \sum_{\mu \subseteq \lambda }
      (-1)^{|\mu |}
      \binom{\lambda }{\mu }_{\theta }
      \prod_{i = 1}^{N}
         \frac{(a + \theta + 1 + \theta (N - i))_{\lambda_{i}}}{(a + \theta + 1 + \theta (N - i))_{\mu_{i}}}
      \frac{P_{\mu }^{\theta }(\xx_N)}{P_{\mu }^{\theta}(1_N)} 
      \prod_{i = 1}^{N}
         \frac{((N + \alpha + 1 - i) \theta )_{\lambda_{i}}}{((N + \alpha + 2 - i) \theta )_{\lambda_{i}}} \\
   &=
   L_{\lambda }^{a + \theta }(\xx_{N} ;\theta^{-1} )
   c(\lambda, N,\theta \,; \alpha).
\end{align*}
%
Hence, \tref{thm:Laguerre intertwine} is equivalent to \lref{c:43} under the assumptions of \tref{t:22} and the Rodrigue type formula \eqref{:A3}.

From the formula (\ref{:A5}), we also have a parameter shift formula of multivariate hypergeometric functions associated with Jack polynomials (see \cite{BaF97}, \cite{Kan93}, \cite[Chapter 3]{ViK95}):
\begin{align}\label{:A7}
  & {_{p}F_q}^{}\Big(\begin{matrix} a_1 , \ldots, a_p  \\ b_1 , \ldots , b_q  \end{matrix};\xx_{N} \Big):=
    \sum_{n = 0}^{\infty }\sum_{|\lambda | = n}
    \frac{(a_{1} )_{\lambda } \cdots (a_p )_{\lambda }}{(b_1 )_{\lambda} \cdots (b_{q} )_{\lambda}}
    \frac{C_{\lambda }^{\theta ^{-1}}(\xx_{N})}{n!}
,\end{align}
where $(a)_{\lambda}=\prod_{i=1}^{l(\lambda )} (a-\theta (i-1))_{\lambda_{i}}$.

\begin{proposition}
We assume that \eqref{:A7} converges.  
For any $1 \leq j \leq q$, under the condition $b_j =(\beta_j +N)\theta $, we have  
\begin{align}
  \Big[ \Lambda_{\theta, \beta_{j}, N}^{N} \Big\{  {_{p}F_q}^{} \Big(\begin{matrix} a_1 , \ldots, a_p  \\ b_1   , \ldots b_q  \end{matrix}; \cdot  \Big) \Big\} \Big]  (\xx_{N})=
     {_{p}F_q}^{}\Big(\begin{matrix} a_1  , \ldots, a_p \\ b_1 , \ldots ,b_{j-1},  b_j+\theta ,b_{j+1} , \ldots ,b_q \end{matrix};\xx_{N} \Big).
  \end{align}  
\end{proposition}

\section{Intertwining of $\beta$-Laguerre Ornstein-Uhlenbeck  processes }\label{appendix:B}
Let $\beta \ge 1 $ and $\alpha > -1$ as before.
We consider the $N$-dimensional stochastic differential equation  
\begin{align}\label{:B1}
  dX_t^{N,i} &=\sqrt{2 X_t^{N, i} } dB_t^i  -X_t^i dt +\frac{\beta }{2 }\Big( \alpha +1 +\sum_{1\le j\le N, j\neq i}\frac{2 X_t^{N, i} }{X_t^{N, i} -X_t^{N,j} } \Big)dt, \qquad 1\le i\le N 
,\end{align}
where $(B^i)_{i=1}^N $ is the $N$-dimensional Brownian motion.
By slightly modifying the proof of \cite[Corollary 6.5]{GrM14}, we see that the equation \eqref{:B1} has a unique strong solution with no collisions and no explosions.
We call the unique solution to \eqref{:B1} the $N$-dimensional  $\beta$-Laguerre Ornstein-Uhlenbeck process for parameter $\alpha$, and define $\mathfrak{T}_{\theta, \alpha,t}^{N}$ as the associated Markov semigroup.
The infinitesimal generator associated with \eqref{:B1} is given by
\begin{align}\notag
  \mathfrak{A}_{\theta, \alpha }^N := D_1^N + \theta(\alpha +1) E_0^N -E_1^N
.\end{align}

The $\beta$-Laguerre Ornstein-Uhlenbeck processes satisfy shifted intertwining relations in analogy to \pref{p:14}.
\begin{proposition}
Assume that $\theta \ge 1/2$ and $\alpha >-1$.  
Then, for any $N\in\N$ and $t\ge 0$, we have 
  \begin{align}\label{:B2}
\mathfrak{T}_{\theta, \alpha ,t}^{N+1 } \LtNNp  =\LtNNp \mathfrak{T}_{\theta, \alpha +1 ,t}^{N } 
  .\end{align}
\begin{proof}
This was essentially proved in \cite{Ass19}. 
Actually, $E_1^{N+1} \LtNNp P_{\lambda }^\theta = \LtNNp E_1^N P_{\lambda }^\theta $ holds from \eqref{:22a} and \eqref{:40c}.
Combining this with \cite[(21)]{Ass19},  we have $ \mathfrak{A }_{\theta, \alpha }^{N+1 } \LtNNp P_{\lambda }^\theta = \LtNNp \mathfrak{A }_{\theta, \alpha+1 }^{N } P_{\lambda }^\theta $.
Furthermore, an exponential estimate of the $\beta$-Laguerre Ornstein-Uhlenbeck processes can be derived from corresponding estimate for the $\beta$-Laguerre processes with a comparison theorem of stochastic differential equations.
Therefore, the argument in \cite{Ass19}, with the above modifications, shows that \eqref{:B2} holds.
\end{proof}
\end{proposition}

The following theorem can be proved by the same method in \sref{s:5} via Jack polynomials.
On the other hand, by using multivariate Laguerre polynomials instead of Jack polynomials, the proof becomes somewhat simpler.
\begin{theorem}
Suppose $\theta \ge 1/2 $ and $\alpha >-1$.
Then, for any $ N\in\N$ and $ t\ge 0$,  we have 
\begin{align}\label{:B3}
\mathfrak{T} _{\theta, \alpha+1, t}^{N} \LtaNN &= \LtaNN \mathfrak{T}_{\theta, \alpha, t}^{N} 
  .\end{align}  
Furthermore, we have 
\begin{align}\label{:B4}
  \mathfrak{T} _{\theta, \alpha, t}^{N+1} \LtaNNp &= \LtaNNp  \mathfrak{T}_{\theta, \alpha, t}^{N} 
.\end{align} 
\begin{proof}
The equation \eqref{:B4} results from \eqref{:B2} and \eqref{:B3} immediately.
To establish \eqref{:B3}, by the same argument as in \lref{l:51} with \eqref{:A2}, we get
\begin{align}\label{:B6}
[\mathfrak{T}_{\theta, \alpha, t}^{N} L_{\lambda }^{\theta(\alpha+1)-1}(\cdot ; \theta^{-1}) ](\xx_N ) &= L_{\lambda }^{\theta(\alpha+1)-1}(\xx_N ;\theta^{-1}) + \int_0^t [\mathfrak{T}_{\theta, \alpha, s}^{N} \mathfrak{A}_{\theta, \alpha }^{N} L_{\lambda }^{\theta(\alpha+1)-1}(\cdot ; \theta^{-1}) ] (\xx_N ) ds
 \\& \notag
=L_{\lambda }^{\theta(\alpha+1)-1}(\xx_N; \theta^{-1}) -|\lambda | \int_0^t [\mathfrak{T}_{\theta, \alpha, s}^{N} L_{\lambda }^{\theta(\alpha+1)-1} (\cdot ; \theta^{-1})](\xx_N ) ds
,\end{align}  
which yields 
\begin{align}\label{:B7}
  [\mathfrak{T}_{\theta, \alpha, t}^{N} L_{\lambda }^{\theta(\alpha+1)-1}(\cdot ; \theta^{-1})] (\xx_N ) &= e^{-|\lambda|t } L_{\lambda }^{\theta(\alpha+1)-1}(\xx_N; \theta^{-1})
.\end{align} 
Combining the equation \eqref{:B7} with \eqref{eq:Laguerre intertwine}, we obtain 
\begin{align}\label{:B3a}
 [\LtaNN  \mathfrak{T}_{\theta, \alpha, t}^{N} L_{\lambda }^{\theta(\alpha+1)-1}(\cdot ; \theta^{-1}) ](\xx_N ) &=e^{-|\lambda|t } L_{\lambda }^{\theta(\alpha+2)-1}(\xx_N; \theta^{-1}) c(\lambda,N,\theta; \alpha)
\\& \notag
=[\mathfrak{T} _{\theta, \alpha+1, t}^{N} L_{\lambda }^{\theta(\alpha+2)-1}(\cdot ; \theta^{-1})  ] (\xx_N) c(\lambda,N,\theta; \alpha)
\\& \notag
=[\mathfrak{T} _{\theta, \alpha+1, t}^{N} \LtaNN L_{\lambda }^{\theta(\alpha+1)-1}(\cdot ; \theta^{-1})](\xx_N )
.\end{align}
Because any symmetric polynomial is a linear combination of multivariate Laguerre polynomials, the equation \eqref{:B3a} enables us to apply the same argument in the proof of \tref{t:17}, and thus we conclude \eqref{:B3}.
\end{proof}
\end{theorem}

\begin{remark}
When $\theta =1$, the equations \eqref{:B3} and \eqref{:B4} were proved in \cite{BuK} by a different approach.
\end{remark}

\section*{Acknowledgement}
K. Y. is supported by JSPS KAKENHI Grant Number 21K13812.
S. G. is supported by JST CREST Grant Number JP19209317 and JSPS KAKENHI Grant Number
21K13808.

\end{document}